\documentclass[11pt,a4paper]{article}
\usepackage{enumerate, amsmath, amsthm}
\usepackage{amssymb, amsfonts, mathrsfs}
\usepackage{comment}
\usepackage{array}

\usepackage[dvipdfmx]{graphicx}

\theoremstyle{definition}
\newtheorem{thm}{Theorem}

\newtheorem{lem}[thm]{Lemma}
\newtheorem{cor}[thm]{Corollary}

\bfseries\normalfont

\begin{document}

\title{Proper $3$-orientations of bipartite planar graphs with minimum degree at least $3$}

\author{
Kenta Noguchi\thanks{Department of Information Sciences, 
Tokyo University of Science, 
2641 Yamazaki, Noda, Chiba 278-8510, Japan.
Email: {\tt noguchi@rs.tus.ac.jp}}
}

\date{}
\maketitle

\noindent
\begin{abstract}
In this short note, 
we show that every bipartite planar graph with minimum degree at least $3$ has proper orientation number at most $3$. 
\end{abstract}


\section{Introduction}
\label{introsec}
For basic terminology in graph theory undefined in this paper, refer to \cite{BM}. 
Let $G$ be a simple graph. 
An {\em orientation} $\sigma$ of $G$ is a digraph 
obtained from $G$ by replacing each edge by exactly one of the two possible arcs 
with the same end-vertices. 
Let $E_{\sigma}(G)$ be the resulting arc set. 
For $v \in V(G)$, 
the {\em indegree} of $v$ in $\sigma$, denoted by $d^-_{\sigma}(v)$, 
is the number of arcs in $E_{\sigma}(G)$ incoming to $v$. 
We denote by $\Delta^-(\sigma)$ the {\em maximum indegree} of $\sigma$. 
An orientation $\sigma$ with $\Delta^-(\sigma)$ at most $k$ is a {\em $k$-orientation}. 
An orientation $\sigma$ of $G$ is {\em proper} 
if $d^-_{\sigma}(u) \ne d^-_{\sigma}(v)$ for every $uv \in E(G)$. 
The {\em proper orientation number} of $G$, 
denoted by $\overrightarrow{\chi}(G)$, 
is the minimum integer $k$ such that $G$ admits a proper $k$-orientation. 
Proper orientation number is defined by Ahadi and Dehghan \cite{AD}, 
and see the related research \cite{ACdHM, AHLS, KMdML}. 
Knox et al. \cite{KMdML} showed the following. 

\begin{thm}[\cite{KMdML}, Theorem 3] 
\label{thmA}
Let $G$ be a $3$-connected bipartite planar graph. 
Then $\overrightarrow{\chi}(G) \le 5$. 
\end{thm}

The main theorems of this paper are as follows. 
The {\em maximum average degree} $\mbox{Mad}(G)$ of a graph $G$ is defined as 
\begin{eqnarray*}
\mbox{Mad}(G) = \mbox{max} \bigg\{ \frac{2|E(H)|}{|V(H)|}: H~ \mbox{is a subgraph of}~ G\bigg\}.
\end{eqnarray*}

\begin{thm}
\label{mainthm1}
Let $k$ be a positive integer and 
$G = G[X, Y]$ be a bipartite graph with $\mbox{Mad}(G) \le 2k$. 
If $\deg(x) \ge k+1$ for every $x \in X$, 
then $\overrightarrow{\chi}(G) \le k+1$. 
\end{thm}

Theorem \ref{mainthm1} partially answers Problem 5 in \cite{ACdHM}, 
which asks whether $\overrightarrow{\chi}(G)$ can be bounded by a function of $\mbox{Mad}(G)$. 
Using Theorem \ref{mainthm1}, 
we can state the following. 
The {\em minimum degree} of $G$ is denoted by $\delta(G)$. 

\begin{thm}
\label{mainthm2}
Let $G$ be a bipartite planar graph with $\delta(G) \ge 3$. 
Then $\overrightarrow{\chi}(G) \le 3$. 
\end{thm}

Both of the bounds of $\delta(G)$ and $\overrightarrow{\chi}(G)$ in Theorem \ref{mainthm2} are tight; see Theorems \ref{mainthm3} and \ref{mainthm4}, respectively. 
The following corollary immediately follows from Theorem \ref{mainthm2}, 
which is the improvement of Theorem \ref{thmA}. 

\begin{cor}
Let $G$ be a $3$-connected bipartite planar graph. 
Then $\overrightarrow{\chi}(G) \le 3$. 
\end{cor}

Furthermore, for {\em quadrangulations}, 
which are bipartite planar graphs whose each face is bounded by a $4$-cycle, 
the orientation number is completely determined as follows. 

\begin{thm}
\label{mainthm3}
Let $G$ be a quadrangulation with $\delta(G) \ge 3$. 
Then $\overrightarrow{\chi}(G) = 3$ unless $G$ is isomorphic to the $3$-cube $Q_3$. 
If $G$ is isomorphic to $Q_3$, then $\overrightarrow{\chi}(G) = 2$.
\end{thm}

Finally, we show the tightness of the minimum degree condition in Theorem \ref{mainthm2}. 
Note that, when $\delta(G) = 1$, Araujo et al. \cite[Corollary 2]{ACdHM} showed that there exists a tree $T$ with $\overrightarrow{\chi}(T) = 4$. 

\begin{thm}
\label{mainthm4}
There exist bipartite planar graphs $G$ with $\delta(G) = 2$ and $\overrightarrow{\chi}(G) = 4$. 
\end{thm}

\section{Proofs of main theorems}
\label{Proofs}

The proof of Theorem \ref{mainthm1} is very simple. 

\begin{proof}[Proof of Theorem \ref{mainthm1}] 
Let $G = G[X, Y]$ be a bipartite graph with $\mbox{Mad}(G) \le 2k$ 
satisfying $\deg(x) \ge k+1$ for every $x \in X$. 
By Hakimi's result \cite{Ha}, 
every graph $G$ with $\mbox{Mad}(G) \le 2k$ has a $k$-orientation $\sigma$. 
Let $S_{\sigma} = \{ (u, v) \in E_{\sigma}(G) \mid u \in X \}$. 
For every $x \in X$, 
since $\deg(x) \ge k+1$, 
we can switch the orientation of some edges in $S_{\sigma}$ 
so that the indegree of $x$ is exactly $k+1$. 
In the resulting orientation $\sigma'$ of $G$, 
observe that $d^-_{\sigma'}(y) \le k$ for every $y \in Y$. 
Thus, $\sigma'$ is a proper $(k+1)$-orientation of $G$. 
\end{proof}

To prove Theorem \ref{mainthm2}, 
we use the following well-known lemma. 

\begin{lem}
\label{lem1}
Let $G$ be a bipartite planar graph with $n \ge 4$ vertices. 
Then $|E(G)| \le 2n-4$ with equality holding if and only if $G$ is a quadrangulation. 
\end{lem}

\begin{proof}[Proof of Theorem \ref{mainthm2}]
Let $G$ be a bipartite planar graph with $\delta(G) \ge 3$. 
By Lemma \ref{lem1}, $G$ has maximum average degree less than $4$. 
Thus, $\overrightarrow{\chi}(G) \le 3$ by Theorem \ref{mainthm1}. 
\end{proof}

\begin{proof}[Proof of Theorem \ref{mainthm3}]
Let $G[X, Y]$ be a quadrangulation with $\delta(G) \ge 3$, 
where $|X| \ge |Y|$. 
Note that $n = |X|+|Y| \ge 8$ 
and $|Y| \ge 3$. 
By Theorem \ref{mainthm2}, 
$\overrightarrow{\chi}(G) \le 3$. 
We show that $G$ does not admit a proper $2$-orientation 
or is isomorphic to $Q_3$. 
Since $|E(G)| = 2n-4$ by Lemma \ref{lem1}, 
if there exists a proper $2$-orientation $\sigma$ of $G$, 
then the number of vertices of indegree $2$ in $\sigma$ is at least $n-4$. 
Since two vertices of indegree $2$ in $\sigma$ cannot be adjacent in $G$, 
$|Y| \le 4$. 
If $|Y| = 3$, 
then three vertices $x_1, x_2, x_3$ in $X$ must be adjacent to all vertices in $Y$, 
and hence there exists a complete bipartite graph $K_{3, 3}$, 
contrary to the planarity of $G$. 
If $|Y| = 4$, let $Y = \{ y_1, y_2, y_3, y_4\}$. 

Case 1. 
Suppose that there exist two vertices $x_1, x_2 \in X$ 
such that each of them is adjacent to the same three vertices in $Y$, 
say $y_1 , y_2, y_3$ 
(see Figure \ref{fig:1}). 
We may assume that $y_4$ is in the face $f = x_1y_1x_2y_2$ by symmetry. 
Then the other vertices in $X$ must be in $f$ to be adjacent to at least $3$ vertices in $Y$ 
by the planarity of $K_{2, 3}$. 
So the degree of $y_3$ is exactly $2$, 
which contradicts to $\delta(G) \ge 3$. 

\begin{figure}[tb]
 \centering
  \includegraphics[width=3cm]{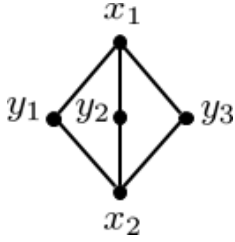}
  \caption{The complete bipartite graph $K_{2,3}$.}
 \label{fig:1}
\end{figure}

Case 2. 
Suppose that there exists no such pair of vertices in $X$. 
In this case one can see that $|X| = 4$ 
and that $G$ is isomorphic to $Q_3$. 
It is easy to show that $\overrightarrow{\chi}(Q_3) = 2$ 
(see also \cite[Proposition 2]{ACdHM}). 
\end{proof}

\begin{proof}[Proof of Theorem \ref{mainthm4}]
Consider the bipartite planar graph (quadrangulation) $G$ defined as follows (see Figure \ref{fig:2}). 
Let 
$V(G) = \{a_{ijk} \} \cup \{b_{ij}, c_{ij} \} \cup \{d_{ik} \} \cup \{p_i, q_i \} \cup \{ s, t \}$, 
where $1\le i, j \le 4, ~1 \le k \le 7$, 
and 
\begin{eqnarray*}
E(G) &=& \{a_{ijk}b_{ij}, a_{ijk}c_{ij} \mid 1\le i, j \le 4, ~1 \le k \le 7 \} \cup\\
& & \{b_{ij}p_i, b_{ij}q_i, c_{ij}p_i, c_{ij}q_i \mid 1\le i, j \le 4 \} \cup\\
& & \{d_{ik}p_i, d_{ik}q_i \mid 1\le i \le 4, ~1 \le k \le 7 \} \cup\\
& & \{p_is, p_it, q_is, q_it \mid 1\le i \le 4 \}. 
\end{eqnarray*}

\begin{figure}[tb]
 \centering
  \includegraphics[width=16cm]{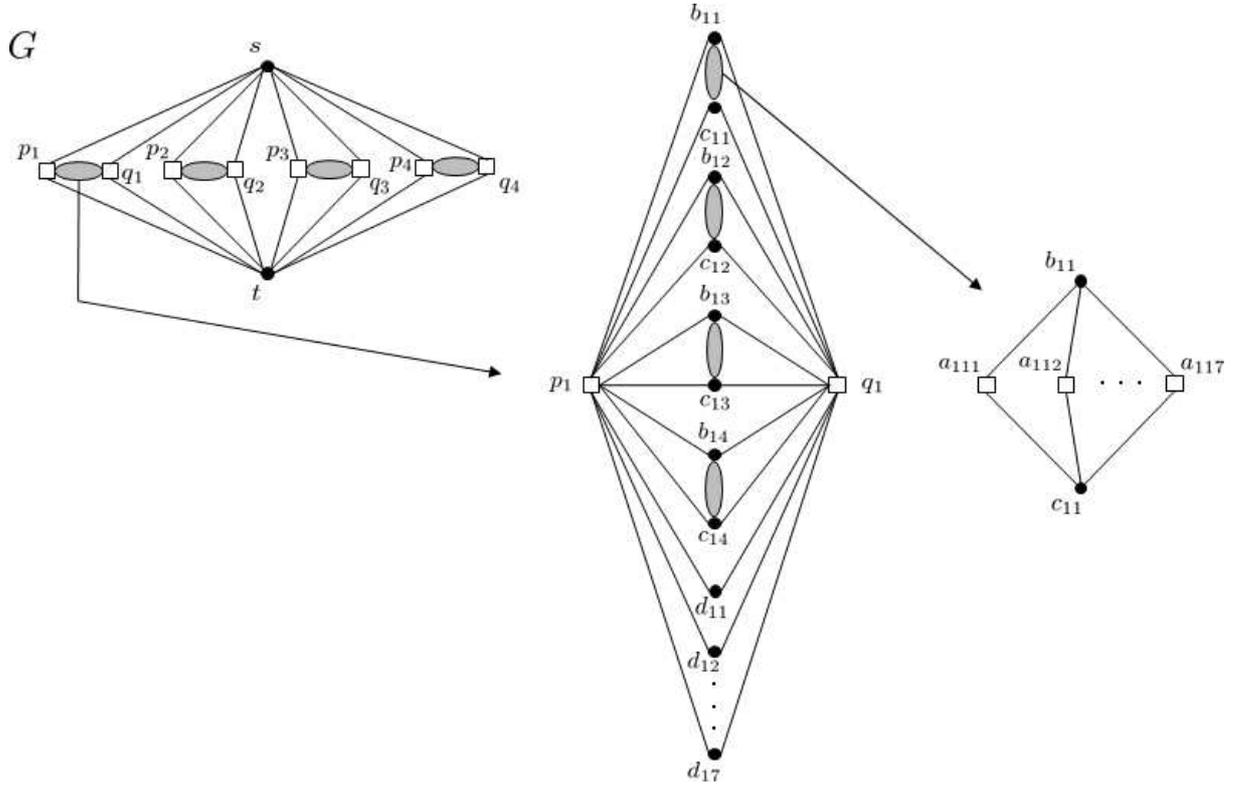}
  \caption{The bipartite planar graph $G$.}
 \label{fig:2}
\end{figure}

Suppose that $\overrightarrow{\chi}(G) \le 3$ and let $\sigma$ be a proper $3$-orientation of $G$. 
Thoughout the proof, an indegree means one in $\sigma$. 

Step 1.
Consider the graph induced by $\{a_{ijk} \} \cup \{b_{ij}, c_{ij} \}$ for fixed $i, j$. 
Since $d^-_{\sigma}(b_{ij}) \le 3$ and $d^-_{\sigma}(c_{ij}) \le 3$, 
there exists a vertex of indegree $2$ in $\{a_{ijk} \}$. 
So $b_{ij}$ and $c_{ij}$ cannot be indegree $2$ for all $1\le i, j \le 4$. 

Step 2.
Consider the graph induced by $\{a_{ijk} \} \cup \{b_{ij}, c_{ij} \} \cup \{d_{ik} \} \cup \{p_i, q_i \}$ for a fixed $i$. 
If all vertices in $\{b_{ij}, c_{ij} \}$ have indegree at most $1$, 
then at least one of the indegrees of $p_i$ and $q_i$ must be at least $4$, a contradiction. 
By this fact and Step 1, 
there exists a vertex of indegree $3$ in $\{ b_{ij}, c_{ij} \}$. 
So $p_i$ and $q_i$ cannot be indegree $3$ for all $1\le i \le 4$. 
Moreover, 
since there exists a vertex of indegree $2$ in $\{d_{ik} \}$ by the same argument in Step 1, 
$p_i$ and $q_i$ cannot be indegree $2$ for all $1\le i \le 4$. 

Step 3. 
By Steps 1 and 2, all of $p_1, \ldots, p_4, q_1, \ldots, q_4$ have indegree at most $1$. 
This leads to the fact that at least one of the indegrees of $s$ and $t$ must be at least $4$, 
a contradiction. 

Thus, $\overrightarrow{\chi}(G) \ge 4$. 
It is easy to show that $\overrightarrow{\chi}(G) \le 4$. 
(An infinite family of such graphs are easy to construct; add some vertices $a_{118}, \ldots, a_{11k}$ of degree $2$ to $G$ for example.) 
\end{proof}

\section*{Acknowledgements}
The author's work was partially supported by 
JSPS KAKENHI Grant Number 17K14239.

\end{document}